\documentclass[12pt]{article}

\usepackage{amssymb}
\usepackage{amsthm}
\usepackage{fleqn}
\usepackage{graphicx}
\usepackage{epstopdf}
\usepackage{amsfonts}
\usepackage[numbers]{natbib}
\usepackage[colorlinks,citecolor=blue,urlcolor=blue]{hyperref}
\usepackage{amsmath}
\usepackage{geometry}
\usepackage{pifont}
\newtheorem{theorem}{Theorem}[section]

\begin{document}

\title{Solving a class of zero-sum stopping game with regime switching
\thanks{This work was supported by
the National Key R\&D Program of China (2022ZD0120001),
the National Natural Science Foundation of China (11801072, 11971267),
and the Fundamental Research Funds for the Central Universities (2242021R41175).}}

\author{Siyu Lv\thanks{School of Mathematics, Southeast University,
Nanjing 211189, China (lvsiyu@seu.edu.cn).}
\and
Xiao Yang\thanks{School of Mathematics, Southeast University,
Nanjing 211189, China (yangxiaoseu@163.com).}}

\date{}

\maketitle

\begin{abstract}
This paper studies a class of zero-sum stopping game in a regime switching model.
A verification theorem as a sufficient criterion for Nash equilibriums is established
based on a set of variational inequalities (VIs). Under an appropriate regularity
condition for solutions to the VIs, a suitable system of algebraic equations is derived
via the so-called smooth-fit principle. Explicit Nash equilibrium stopping rules of
threshold-type for the two players and the corresponding value function of the game
in closed form are obtained. Numerical experiments are reported to demonstrate the
dependence of the threshold levels on various model parameters. A reduction to the
case with no regime switching is also presented as a comparison.
\end{abstract}


\section{Introduction}

Optimal stopping is concerned with the problem that among all possible choices
of stopping times, we are seeking for an optimal one, which gives the best result
in the sense of expectation. To solve optimal stopping problems,
the \emph{variational inequality} (VI) approach has been extensively employed
because it provides some sufficient conditions that are easy to verify
and typically leads to ordinary or partial differential equations
that can be solved analytically or numerically; see, for example,
{\O}ksendal \cite[Chapter 10]{Oksendal2003} and Pham \cite[Chapter 5]{Pham2009}.
On the other hand, optimal stopping has a wide applications in many fields,
such as stock selling ({\O}ksendal \cite[Examples 10.2.2 and 10.4.2]{Oksendal2003})
and option pricing (McKean \cite{Mckean1965}). The game problems of optimal stopping
were initially suggested and investigated by Friedman \cite{F1973} (zero-sum case)
and Bensoussan and Friedman \cite{BF1977} (nonzero-sum case) also using the VI approach.
Since then, many interesting works were motivated along this line;
see Akdim et al. \cite{AOT2006}, De Angelis et al. \cite{DFM2018},
Lv et al. \cite{LWZ2022}, and so on.

The regime switching model is a two-component process $(X_{t},\alpha_{t})$ in which
the first component $X_{t}$ evolves according to a continuous diffusion process
whose drift and diffusion coefficients depend on the regime of $\alpha_{t}$,
where $\alpha_{t}$ is generally assumed to be a finite-state Markov chain.
As a result, the regime switching model exhibits a ``hybrid" feature and has
the ability to capture more directly the discrete events that are less frequent
(occasional) but nevertheless more significant to longer-term system behavior.
In addition, owing to the drift and diffusion coefficients taking only finite
number of values, the regime switching model also has the tractability which enables
feasible numerical schemes to be developed. For more analysis and applications of
regime switching models, one is referred to the monographs by Yin and Zhu \cite{YinZhu2010}
and Yin and Zhang \cite{YinZhang2013}.

Optimal stopping problems for regime switching models have been studied by many researchers
under various contexts and different formulations. Here, we only name a few that are closely
related to our work. Zhang \cite{Zhang2001} considered an optimal selling problem of a stock,
whose price satisfies a geometric Brownian motion modulated by a two-state Markov chain,
through a \emph{two-point boundary-value differential equation} (TPBVDE) approach. The optimal
stopping rule is of threshold-type and to stop whenever the stock price reaches two
pre-defined (lower and upper) bounds. Guo and Zhang \cite{GuoZhang2005} dealt with a similar
problem as \cite{Zhang2001} but instead adopting the so-called \emph{smooth-fit principle}.
The optimal stopping rule is also of threshold-type such that there exist two threshold levels
corresponding to the two states of the Markov chain. Recently, these two methods were
combined together by Zhang and Zhou \cite{ZhangZhou2009} to solve a stock loan valuation
problem with regime switching.

In this paper, we study a class of zero-sum stopping game in a regime switching model.
In order to highlight the main idea and obtain a closed-form solution, we consider
a simple but illustrative formulation, i.e., the state process is described by a scaled
Brownian motion modulated by a two-state Markov chain and the payoffs for the two players
to optimize are linear. Compared with the optimal stopping problems (not games) considered
in \cite{Zhang2001,GuoZhang2005,ZhangZhou2009}, the analysis of the current paper
is more involved. In particular, the partition of \emph{stopping region} and
\emph{continuation region} and the resolution of the VIs are more complicated.
To the best of our knowledge, this paper is the first attempt to establish a theoretical
framework and an analytical approach for such kind of problem. The framework and approach
proposed could be used as a guide for treating the problems with more general models or
more difficult situations.

This paper mainly consists of three parts. The first part is to establish
a verification theorem as a sufficient criterion associated with a set of VIs for Nash
equilibriums. It is proved that according to the VIs, a Nash equilibrium for the two players
can be constructed in terms of the stopping region and continuation region and the solution
to the VIs coincides with the corresponding value function of the game.
Then, in the second part, we adopt the smooth-fit principle to solve the VIs.
Some delicate analysis and matrix manipulation are carried out to obtain explicit Nash
equilibrium stopping rules of threshold-type and the value function in closed-form.
Of course, the threshold levels should depend on the state of the Markov chain.
Finally, in the third part, numerical experiments are reported to demonstrate
the dependence of the threshold levels on various model parameters.
Moreover, a reduction to the case when there is no regime switching is also presented.

It is emphasized that in the verification theorem, an \emph{appropriate}
regularity condition (see condition (a) of Theorem \ref{VT theorem})
is so crucial in that, with the help of the smooth-fit principle, a system of algebraic
equations can be derived which is suitable in the sense that the number of equations
is \emph{equal} to that of undetermined parameters. Consequently, one needs only to solve
an algebraic system in order to identify a Nash equilibrium and the value function.
If, on the other hand, the regularity condition is set too strong or too weak
so that the resulting algebraic system has no solution or infinitely many solutions,
then the smooth-fit technique would collapse.

We would like to point out that closed-form solutions in stochastic control problems
are rarely obtainable. A closed-form solution is desirable in practice because
it provides a clear picture on dependence of model parameters and could be useful
for related computational methods to be developed. This paper reports a closed-form
solution to a class of stopping game with regime switching, which adds to the list
of ``solvable" stochastic control problems in the literature.

The rest of this paper is organized as follows. Section \ref{section VT} formulates
the problem and establishes the verification theorem. Section \ref{section SF} obtains
an explicit Nash equilibrium and the corresponding value function in closed-form.
Section \ref{section NR} reports numerical experiments to examine the dependence
of the threshold levels and presents a reduction to the case with no regime switching.
Finally, Section \ref{section CR} concludes the paper.

\section{Verification theorem}\label{section VT}

Let $(\Omega,\mathcal{F},P)$ be a fixed probability space on which
a one-dimensional standard Brownian motion $W_{t}$, $t\geq0$,
and a two-state Markov chain $\alpha_{t}\in\mathcal{M}=\{1,2\}$, $t\geq0$,
are defined. The generator of $\alpha_{t}$ is given by
$$
\left[
  \begin{array}{cc}
    -\lambda_{1} & \lambda_{1} \\
    \lambda_{2} & -\lambda_{2} \\
  \end{array}
\right],
$$
for some $\lambda_{1}>0$ and $\lambda_{2}>0$. Assume that $W$ and $\alpha$
are independent. Let $\{\mathcal{F}_{t}\}_{t\geq0}$ be the natural filtration
of $W$ and $\alpha$.

In this paper, $f^{\prime}(x)$ (respectively, $f^{\prime\prime}(x)$) denotes the
first (respectively, second order) derivative of a function $f$ with respect to $x$.
$C^{1}$ (respectively, $C^{2}$) denotes the space of functions whose first order
(respectively, second order) derivatives are continuously differentiable.
$\partial D$ denotes the boundary of a region $D$.

Let the two players in the game be labeled by Player 1 and Player 2.
The one-dimensional state process $X$ is described by
\begin{equation}\label{state equation}
\begin{aligned}
X_{t}=x+\int_{0}^{t}\sigma(\alpha_{s})dW_{s},\quad t\geq 0,
\end{aligned}
\end{equation}
where $\sigma(i)$, $i\in\mathcal{M}$, are positive constants.
The objective functional for Player 1 to minimize and Player 2 to maximize is given by
\begin{equation}\label{payoff}
\begin{aligned}
J(x,i;\tau_{1},\tau_{2})
=E[e^{-r\tau_{1}}(X_{\tau_{1}}-K(\alpha_{t}))1_{\{\tau_{1}<\tau_{2}\}}
+e^{-r\tau_{2}}(X_{\tau_{2}}-\widetilde{K}(\alpha_{t}))1_{\{\tau_{2}\leq\tau_{1}\}}],
\end{aligned}
\end{equation}
where $K(i)$ and $\widetilde{K}(i)$, $i\in\mathcal{M}$, are constants with
$K(i)<\widetilde{K}(i)$, $r>0$ is the discount factor, and $\tau_{1}$ and $\tau_{2}$
are $\mathcal{F}_{t}$-stopping times chosen by Player 1 and Player 2, respectively.
The aim is to find a Nash equilibrium $(\tau_{1}^{*},\tau_{2}^{*})$ such that
\begin{equation*}
\begin{aligned}
J(x,i;\tau_{1}^{*},\tau_{2})
\leq J(x,i;\tau_{1}^{*},\tau_{2}^{*})
\leq J(x,i;\tau_{1},\tau_{2}^{*}).
\end{aligned}
\end{equation*}
If such a Nash equilibrium exists, we denote
$v(x,i)=J(x,i;\tau_{1}^{*},\tau_{2}^{*})$ as the
corresponding value function of the game.

Denote
\begin{equation*}
\begin{aligned}
D_{1,i}=\{v(x,i)-(x-K(i))<0\},
\end{aligned}
\end{equation*}
and
\begin{equation*}
\begin{aligned}
D_{2,i}=\{v(x,i)-(x-\widetilde{K}(i))>0\}.
\end{aligned}
\end{equation*}
In fact, $D_{1,i}$ (respectively, $D_{2,i}$) is the so-called continuation region
for Player 1 (respectively, Player 2) when the Markov chain is at regime $i$.

In the following, we state and prove the verification theorem for our zero-sum
stopping game problem with regime switching.
\begin{theorem}\label{VT theorem}
Let $v(x,i)$, $i\in\mathcal{M}$, be a real-valued function satisfying
the following conditions:

{\rm(a)} For $i\in\mathcal{M}$ and $k=1,2$,
\begin{equation*}
\begin{aligned}
v(\cdot,i)\in&C^{2}(R\backslash (\cup_{i\in\mathcal{M},k=1,2}\partial D_{k,i}))\cap C^{1}(R).
\end{aligned}
\end{equation*}

{\rm(b)} For $i\in\mathcal{M}$ and $x\in R$,
\begin{equation*}
\begin{aligned}
x-\widetilde{K}(i)\leq v(x,i)\leq x-K(i).
\end{aligned}
\end{equation*}

{\rm(c)} For $i,j\in\mathcal{M}$ with $i\neq j$ and $x\in D_{1,i}$,
\begin{equation*}
\begin{aligned}
\min\bigg\{rv(x,i)-\frac{1}{2}\sigma^{2}(i)v^{\prime\prime}(x,i)
-\lambda_{i}[v(x,j)-v(x,i)],&\\
v(x,i)-(x-\widetilde{K}(i))\bigg\}=0.&
\end{aligned}
\end{equation*}

{\rm(d)} For $i,j\in\mathcal{M}$ with $i\neq j$ and $x\in D_{2,i}$,
\begin{equation*}
\begin{aligned}
\max\bigg\{rv(x,i)-\frac{1}{2}\sigma^{2}(i)v^{\prime\prime}(x,i)
-\lambda_{i}[v(x,j)-v(x,i)],&\\
v(x,i)-(x-K(i))\bigg\}=0.&
\end{aligned}
\end{equation*}
Define $\tau_{1}^{*}$ and $\tau_{2}^{*}$ as
\begin{equation}\label{optimal 1}
\begin{aligned}
\tau_{1}^{*}=&\inf\{t\geq0|X_{t}\notin D_{1,\alpha_{t}}\},
\end{aligned}
\end{equation}
and
\begin{equation}\label{optimal 2}
\begin{aligned}
\tau_{2}^{*}=&\inf\{t\geq0|X_{t}\notin D_{2,\alpha_{t}}\}.
\end{aligned}
\end{equation}
Then, $(\tau_{1}^{*},\tau_{2}^{*})$ is a Nash equilibrium for the two players
and $v(x,i)$, $i\in\mathcal{M}$, is the corresponding value function of the game.
\end{theorem}
\begin{proof}
We first note that, from condition (a), $v(\cdot,i)$, $i\in\mathcal{M}$, is only
$C^{1}$ and not necessarily $C^{2}$ at the boundaries of the continuation regions.
Actually, the regularity condition (a) is set in the present form to ensure that
the system of algebraic equations (\ref{SF-v1-a1-1})-(\ref{SF-v2-b2-2})
(resulted from the smooth-fit principle) has the same number of equations and unknowns
(see Section \ref{section SF}). However, in this theorem, in order to apply It\^{o}'s formula,
we still need $C^{2}$-smoothness at the boundaries.
This can be remedied by the \emph{smooth approximation for variational inequalities}
developed by {\O}ksendal \cite[Theorem 10.4.1 and Appendix D]{Oksendal2003}.
Thus, for convenience, here we directly consider $v(\cdot,i)$, $i\in\mathcal{M}$,
to be $C^{2}$ on the whole space; see Guo and Zhang \cite[Theorem 3.1]{GuoZhang2004},
Guo and Zhang \cite[Theorem 2]{GuoZhang2005}, and A\"{\i}d et al. \cite[Theorem 1]{ABCCV2020}
for a similar treatment.

We first prove that $v(x,i)\leq J(x,i;\tau_{1},\tau_{2}^{*})$,
where $\tau_{1}$ is an arbitrary stopping time chosen by Player 1
and $\tau_{2}^{*}$ is given by (\ref{optimal 2}).
By applying It\^{o}'s formula to $e^{-rt}v(X_{t},\alpha_{t})$
between 0 and $\tau=\tau_{1}\wedge\tau_{2}^{*}$, we have
\begin{equation*}
\begin{aligned}
&E[e^{-r\tau}v(X_{\tau},\alpha_{\tau})]-v(x,i)\\
=&E\bigg[\int_{0}^{\tau}e^{-rt}\bigg\{-rv(X_{t},\alpha_{t})
+\frac{1}{2}\sigma^{2}(\alpha_{t})v^{\prime\prime}(X_{t},\alpha_{t})\\
&+\lambda_{\alpha_{t}}[v(X_{t},j)-v(X_{t},\alpha_{t})]\bigg\}dt\bigg]\geq0,
\quad j\in\mathcal{M},\quad j\neq\alpha_{t},
\end{aligned}
\end{equation*}
where the inequality is due to condition (d) by noting that
$X_{t}\in D_{2,\alpha_{t}}$ before $\tau$ (or, $\tau_{2}^{*}$).

That is
\begin{equation}\label{VT proof 1}
\begin{aligned}
v(x,i)\leq E[e^{-r\tau}v(X_{\tau},\alpha_{\tau})].
\end{aligned}
\end{equation}
If $\tau=\tau_{1}$, then condition (b) implies
\begin{equation}\label{VT proof 2}
\begin{aligned}
v(X_{\tau},\alpha_{\tau})\leq X_{\tau}-K(\alpha_{\tau}).
\end{aligned}
\end{equation}
If $\tau=\tau_{2}^{*}$, then the definition (\ref{optimal 2}) of $\tau_{2}^{*}$ yields
\begin{equation}\label{VT proof 3}
\begin{aligned}
v(X_{\tau},\alpha_{\tau})=X_{\tau}-\widetilde{K}(\alpha_{\tau}).
\end{aligned}
\end{equation}
Combining (\ref{VT proof 1})-(\ref{VT proof 3}) and recalling
the definition of objective functional (\ref{payoff}), we have
\begin{equation*}
\begin{aligned}
v(x,i)\leq J(x,i;\tau_{1},\tau_{2}^{*}).
\end{aligned}
\end{equation*}
The proof of
\begin{equation*}
\begin{aligned}
v(x,i)\geq J(x,i;\tau_{1}^{*},\tau_{2}),
\end{aligned}
\end{equation*}
where $\tau_{1}^{*}$ is given by (\ref{optimal 1}) and
$\tau_{2}$ is an arbitrary stopping time chosen by Player 2,
is similar, and the proof of
\begin{equation*}
\begin{aligned}
v(x,i)=J(x,i;\tau_{1}^{*},\tau_{2}^{*})
\end{aligned}
\end{equation*}
is analogous to the above except that all the inequalities become equalities.
\end{proof}

\section{Smooth-fit and explicit solution}\label{section SF}

In this section, we apply the verification theorem (Theorem \ref{VT theorem})
together with the smooth-fit principle to find an explicit Nash equilibrium
and the corresponding value function in closed-form.

Intuitively, Player 1 (respectively, Player 2) prefers the state process $X_{t}$
to reach a low (respectively, high) level to optimize its own interest.
Note that $(X_{t},\alpha_{t})$ is a joint Markov process, hence it is natural
and reasonable to consider a kind of \emph{threshold-type} and \emph{regime-dependent}
stopping rules for the two players, which can be determined by four
constants $a_{1},a_{2},b_{1},b_{2}$, as follows:
(\expandafter{\romannumeral1}) Player 1 stops only if $X_{t}$ falls below
$a_{1}$ (respectively, $a_{2}$) when $\alpha_{t}=1$ (respectively, $\alpha_{t}=2$),
(\expandafter{\romannumeral2}) Player 2 stops only if $X_{t}$ goes above
$b_{1}$ (respectively, $b_{2}$) when $\alpha_{t}=1$ (respectively, $\alpha_{t}=2$).

Without loss of generality, we assume
\begin{equation}\label{order}
\begin{aligned}
a_{1}<a_{2}<b_{1}<b_{2}.
\end{aligned}
\end{equation}
The other cases with different orders can be treated in the same way.
In fact, (\ref{order}) means that $(-\infty,a_{i})$ and $(a_{i},\infty)$
are the stopping region and continuation region for Player 1, respectively,
and $(b_{i},\infty)$ and $(-\infty,b_{i})$ are the stopping region
and continuation region for Player 2, respectively,
when the Markov chain is at regime $i$.

Note that $(a_{2},b_{1})$ is a common continuation region for both players,
so we have
\begin{equation*}
\left\{
\begin{aligned}
rv(x,1)-\frac{1}{2}\sigma^{2}(1)v^{\prime\prime}(x,1)
-\lambda_{1}[v(x,2)-v(x,1)]=&0,\\
rv(x,2)-\frac{1}{2}\sigma^{2}(2)v^{\prime\prime}(x,2)
-\lambda_{2}[v(x,1)-v(x,2)]=&0.
\end{aligned}
\right.
\end{equation*}
Solving the above equation, it follows that on $(a_{2},b_{1})$:
\begin{equation*}
\begin{aligned}
v(x,1)=\sum_{k=1}^{4}A_{k}e^{\beta_{k}x},
\end{aligned}
\end{equation*}
and
\begin{equation*}
\begin{aligned}
v(x,2)=\sum_{k=1}^{4}B_{k}e^{\beta_{k}x},
\end{aligned}
\end{equation*}
where $\beta_{k}$, $k=1,2,3,4$, are solutions to the following equation:
\begin{equation*}
\begin{aligned}
\frac{\sigma^{2}(1)\sigma^{2}(2)}{4}
\bigg[\bigg(\beta^{2}-\frac{2(r+\lambda_{1})}{\sigma^{2}(1)}\bigg)
\bigg(\beta^{2}-\frac{2(r+\lambda_{2})}{\sigma^{2}(2)}\bigg)
-\frac{4\lambda_{1}\lambda_{2}}{\sigma^{2}(1)\sigma^{2}(2)}\bigg]=0,
\end{aligned}
\end{equation*}
and
$$
B_{k}=\rho_{k}A_{k}
$$
with
\begin{equation*}
\begin{aligned}
\rho_{k}=\frac{1}{\lambda_{1}}\bigg[(r+\lambda_{1})
-\frac{1}{2}\sigma^{2}(1)\beta_{k}^{2}\bigg],\quad k=1,2,3,4.
\end{aligned}
\end{equation*}
On $(-\infty,a_{2})$, Player 1 stops when the Markov chain is at regime 2, thus
\begin{equation*}
\begin{aligned}
v(x,2)=x-K(2).
\end{aligned}
\end{equation*}
On $(a_{1},a_{2})$, Player 1 does not stop when the Markov chain is at regime 1, then
\begin{equation*}
\begin{aligned}
rv(x,1)-\frac{1}{2}\sigma^{2}(1)v^{\prime\prime}(x,1)
-\lambda_{1}[(x-K(2))-v(x,1)]=0,
\end{aligned}
\end{equation*}
which admits a solution
\begin{equation*}
\begin{aligned}
v(x,1)=C_{1}e^{\gamma_{1}x}+C_{2}e^{\gamma_{2}x}+px+q,
\end{aligned}
\end{equation*}
where
\begin{equation*}
\begin{aligned}
p=\frac{\lambda_{1}}{r+\lambda_{1}},\quad
q=-\frac{\lambda_{1}K(2)}{r+\lambda_{1}},
\end{aligned}
\end{equation*}
and $\gamma_{1},\gamma_{2}$ are solutions to the following equation:
\begin{equation*}
\begin{aligned}
(r+\lambda_{1})-\frac{1}{2}\sigma^{2}(1)\gamma^{2}=0.
\end{aligned}
\end{equation*}
On $(-\infty,a_{1})$,
Player 1 stops when the Markov chain is at regime 1, so
\begin{equation*}
\begin{aligned}
v(x,1)=x-K(1).
\end{aligned}
\end{equation*}
Similarly, on $(b_{1},\infty)$,
Player 2 stops when the Markov chain is at regime 1, i.e.,
\begin{equation*}
\begin{aligned}
v(x,1)=x-\widetilde{K}(1).
\end{aligned}
\end{equation*}
On $(b_{1},b_{2})$,
Player 2 still stays in its continuation region
when the Markov chain is at regime 2, then
\begin{equation*}
\begin{aligned}
rv(x,2)-\frac{1}{2}\sigma^{2}(2)v^{\prime\prime}(x,2)
-\lambda_{2}[(x-\widetilde{K}(1))-v(x,2)]=0,
\end{aligned}
\end{equation*}
which admits a solution
\begin{equation*}
\begin{aligned}
v(x,2)=\widetilde{C}_{1}e^{\widetilde{\gamma}_{1}x}
+\widetilde{C}_{2}e^{\widetilde{\gamma}_{2}x}
+\widetilde{p}x+\widetilde{q},
\end{aligned}
\end{equation*}
where
\begin{equation*}
\begin{aligned}
\widetilde{p}=\frac{\lambda_{2}}{r+\lambda_{2}},\quad
\widetilde{q}=-\frac{\lambda_{2}\widetilde{K}(1)}{r+\lambda_{2}},
\end{aligned}
\end{equation*}
and $\widetilde{\gamma}_{1},\widetilde{\gamma}_{2}$
are solutions to the following equation:
\begin{equation*}
\begin{aligned}
(r+\lambda_{2})-\frac{1}{2}\sigma^{2}(2)\widetilde{\gamma}^{2}=0.
\end{aligned}
\end{equation*}
On $(b_{2},\infty)$,
Player 2 stops when the Markov chain is at regime 2, so
\begin{equation*}
\begin{aligned}
v(x,2)=x-\widetilde{K}(2).
\end{aligned}
\end{equation*}
We summarize the above analysis in Table \ref{Table}.
\begin{table}[htbp]
\centering
\caption{$v(x,i)$, $i=1,2$}
\label{Table}
\renewcommand\arraystretch{1.5}
\begin{tabular}{|c|c|c|}
\hline
$x$ & $v(x,1)$ & $v(x,2)$ \\
\hline
$(-\infty,a_{1})$ & $x-K(1)$ & $x-K(2)$ \\
\hline
$(a_{1},a_{2})$ & $\sum_{k=1}^{2}C_{k}e^{\gamma_{k}x}+px+q$ & $x-K(2)$ \\
\hline
$(a_{2},b_{1})$ & $\sum_{k=1}^{4}A_{k}e^{\beta_{k}x}$ & $\sum_{k=1}^{4}B_{k}e^{\beta_{k}x}$ \\
\hline
$(b_{1},b_{2})$ & $x-\widetilde{K}(1)$ &
$\sum_{k=1}^{2}\widetilde{C}_{k}e^{\widetilde{\gamma}_{k}x}+\widetilde{p}x+\widetilde{q}$ \\
\hline
$(b_{2},\infty)$ & $x-\widetilde{K}(1)$ & $x-\widetilde{K}(2)$ \\
\hline
\end{tabular}
\end{table}

Now we apply the smooth-fit principle to $v(x,1)$ and $v(x,2)$ at the boundaries
of continuation regions of the two players, i.e.,
we shall paste $v(x,1)$ continuously differentiable at $a_{1}$, $a_{2}$, $b_{1}$
and paste $v(x,2)$ continuously differentiable at $a_{2}$, $b_{1}$, $b_{2}$.
To be precise:

For $v(x,1)$: at $a_{1}$,
\begin{align}
&a_{1}-K(1)=C_{1}e^{\gamma_{1}a_{1}}+C_{2}e^{\gamma_{2}a_{1}}+pa_{1}+q,\label{SF-v1-a1-1}\\
&1=C_{1}\gamma_{1}e^{\gamma_{1}a_{1}}+C_{2}\gamma_{2}e^{\gamma_{2}a_{1}}+p,\label{SF-v1-a1-2}
\end{align}
at $a_{2}$,
\begin{align}
&C_{1}e^{\gamma_{1}a_{2}}+C_{2}e^{\gamma_{2}a_{2}}+pa_{2}+q
=\sum_{k=1}^{4}A_{k}e^{\beta_{k}a_{2}},\label{SF-v1-a2-1}\\
&C_{1}\gamma_{1}e^{\gamma_{1}a_{2}}+C_{2}\gamma_{2}e^{\gamma_{2}a_{2}}+p
=\sum_{k=1}^{4}A_{k}\beta_{k}e^{\beta_{k}a_{2}},\label{SF-v1-a2-2}
\end{align}
at $b_{1}$,
\begin{align}
&\sum_{k=1}^{4}A_{k}e^{\beta_{k}b_{1}}=b_{1}-\widetilde{K}(1),\label{SF-v1-b1-1}\\
&\sum_{k=1}^{4}A_{k}\beta_{k}e^{\beta_{k}b_{1}}=1.\label{SF-v1-b1-2}
\end{align}
For $v(x,2)$: at $a_{2}$,
\begin{align}
&a_{2}-K(2)=\sum_{k=1}^{4}B_{k}e^{\beta_{k}a_{2}},\label{SF-v2-a2-1}\\
&1=\sum_{k=1}^{4}B_{k}\beta_{k}e^{\beta_{k}a_{2}},\label{SF-v2-a2-2}
\end{align}
at $b_{1}$,
\begin{align}
&\sum_{k=1}^{4}B_{k}e^{\beta_{k}b_{1}}
=\widetilde{C}_{1}e^{\widetilde{\gamma}_{1}b_{1}}
+\widetilde{C}_{2}e^{\widetilde{\gamma}_{2}b_{1}}
+\widetilde{p}b_{1}+\widetilde{q},\label{SF-v2-b1-1}\\
&\sum_{k=1}^{4}B_{k}\beta_{k}e^{\beta_{k}b_{1}}
=\widetilde{C}_{1}\widetilde{\gamma}_{1}e^{\widetilde{\gamma}_{1}b_{1}}
+\widetilde{C}_{2}\widetilde{\gamma}_{2}e^{\widetilde{\gamma}_{2}b_{1}}
+\widetilde{p},\label{SF-v2-b1-2}
\end{align}
at $b_{2}$,
\begin{align}
&\widetilde{C}_{1}e^{\widetilde{\gamma}_{1}b_{2}}
+\widetilde{C}_{2}e^{\widetilde{\gamma}_{2}b_{2}}
+\widetilde{p}b_{2}+\widetilde{q}
=b_{2}-\widetilde{K}(2),\label{SF-v2-b2-1}\\
&\widetilde{C}_{1}\widetilde{\gamma}_{1}e^{\widetilde{\gamma}_{1}b_{2}}
+\widetilde{C}_{2}\widetilde{\gamma}_{2}e^{\widetilde{\gamma}_{2}b_{2}}
+\widetilde{p}=1.\label{SF-v2-b2-2}
\end{align}
Now, we have 12 algebraic equations (\ref{SF-v1-a1-1})-(\ref{SF-v2-b2-2})
for 12 unknowns $a_{1}$, $a_{2}$, $b_{1}$, $b_{2}$, $A_{1}$, $A_{2}$, $A_{3}$, $A_{4}$,
$C_{1}$, $C_{2}$, $\widetilde{C}_{1}$, $\widetilde{C}_{2}$.
Then, we will solve the algebraic system (\ref{SF-v1-a1-1})-(\ref{SF-v2-b2-2})
to get the unknowns explicitly by some delicate matrix manipulation.

In the following, we assume that the related matrices are invertible when needed.
At first, from the equations (\ref{SF-v1-a2-1}), (\ref{SF-v1-a2-2}),
(\ref{SF-v2-a2-1}), (\ref{SF-v2-a2-2}), we have
(recalling that $B_{k}=\rho_{k}A_{k}$, $k=1,2,3,4$)
\begin{equation*}
\begin{aligned}
&\left[
  \begin{array}{c}
    C_{1}e^{\gamma_{1}a_{2}}+C_{2}e^{\gamma_{2}a_{2}}+pa_{2}+q \\
    C_{1}\gamma_{1}e^{\gamma_{1}a_{2}}+C_{2}\gamma_{2}e^{\gamma_{2}a_{2}}+p \\
    a_{2}-K(2) \\
    1 \\
  \end{array}
\right]\\
=&\left[
    \begin{array}{cccc}
      1 & 1 & 1 & 1 \\
      \beta_{1} & \beta_{2} & \beta_{3} & \beta_{4} \\
      \rho_{1} & \rho_{2} & \rho_{3} & \rho_{4} \\
      \beta_{1}\rho_{1} & \beta_{2}\rho_{2} & \beta_{3}\rho_{3} & \beta_{4}\rho_{4} \\
    \end{array}
  \right]
\left[
  \begin{array}{cccc}
    e^{\beta_{1}a_{2}} & 0 & 0 & 0 \\
    0 & e^{\beta_{2}a_{2}} & 0 & 0 \\
    0 & 0 & e^{\beta_{3}a_{2}} & 0 \\
    0 & 0 & 0 & e^{\beta_{4}a_{2}} \\
  \end{array}
\right]
\left[
  \begin{array}{c}
    A_{1} \\
    A_{2} \\
    A_{3} \\
    A_{4} \\
  \end{array}
\right].
\end{aligned}
\end{equation*}
Then we can represent $(A_{1},A_{2},A_{3},A_{4})^{\top}$ as
\begin{equation*}
\begin{aligned}
\left[
  \begin{array}{c}
    A_{1} \\
    A_{2} \\
    A_{3} \\
    A_{4} \\
  \end{array}
\right]
=&\left[
  \begin{array}{cccc}
    e^{-\beta_{1}a_{2}} & 0 & 0 & 0 \\
    0 & e^{-\beta_{2}a_{2}} & 0 & 0 \\
    0 & 0 & e^{-\beta_{3}a_{2}} & 0 \\
    0 & 0 & 0 & e^{-\beta_{4}a_{2}} \\
  \end{array}
\right]
\left[
    \begin{array}{cccc}
      1 & 1 & 1 & 1 \\
      \beta_{1} & \beta_{2} & \beta_{3} & \beta_{4} \\
      \rho_{1} & \rho_{2} & \rho_{3} & \rho_{4} \\
      \beta_{1}\rho_{1} & \beta_{2}\rho_{2} & \beta_{3}\rho_{3} & \beta_{4}\rho_{4} \\
    \end{array}
  \right]^{-1}\\
&\times\left(
\left[
  \begin{array}{cc}
    1 & 1 \\
    \gamma_{1} & \gamma_{2} \\
    0 & 0 \\
    0 & 0 \\
  \end{array}
\right]
\left[
  \begin{array}{cc}
    e^{\gamma_{1}a_{2}} & 0 \\
    0 & e^{\gamma_{2}a_{2}} \\
  \end{array}
\right]
\left[
  \begin{array}{c}
    C_{1} \\
    C_{2} \\
  \end{array}
\right]
+
\left[
  \begin{array}{c}
    pa_{2}+q \\
    p \\
    a_{2}-K(2) \\
    1 \\
  \end{array}
\right]\right).
\end{aligned}
\end{equation*}
From (\ref{SF-v1-a1-1}) and (\ref{SF-v1-a1-2}), we have
\begin{equation}\label{SF-C1 C2}
\begin{aligned}
\left[
  \begin{array}{c}
    C_{1} \\
    C_{2} \\
  \end{array}
\right]
=
\left[
  \begin{array}{cc}
    e^{-\gamma_{1}a_{1}} & 0 \\
    0 & e^{-\gamma_{2}a_{1}} \\
  \end{array}
\right]
\left[
  \begin{array}{cc}
    1 & 1 \\
    \gamma_{1} & \gamma_{2} \\
  \end{array}
\right]^{-1}
\left[
  \begin{array}{c}
    (1-p)a_{1}-q-K(1) \\
    1-p \\
  \end{array}
\right].
\end{aligned}
\end{equation}
So we obtain
\begin{equation}\label{SF-A1 A2 A3 A4 form-1}
\begin{aligned}
\left[
  \begin{array}{c}
    A_{1} \\
    A_{2} \\
    A_{3} \\
    A_{4} \\
  \end{array}
\right]
=&\left[
  \begin{array}{cccc}
    e^{-\beta_{1}a_{2}} & 0 & 0 & 0 \\
    0 & e^{-\beta_{2}a_{2}} & 0 & 0 \\
    0 & 0 & e^{-\beta_{3}a_{2}} & 0 \\
    0 & 0 & 0 & e^{-\beta_{4}a_{2}} \\
  \end{array}
\right]
\left[
    \begin{array}{cccc}
      1 & 1 & 1 & 1 \\
      \beta_{1} & \beta_{2} & \beta_{3} & \beta_{4} \\
      \rho_{1} & \rho_{2} & \rho_{3} & \rho_{4} \\
      \beta_{1}\rho_{1} & \beta_{2}\rho_{2} & \beta_{3}\rho_{3} & \beta_{4}\rho_{4} \\
    \end{array}
  \right]^{-1}\\
&\times\left(
\left[
  \begin{array}{cc}
    1 & 1 \\
    \gamma_{1} & \gamma_{2} \\
    0 & 0 \\
    0 & 0 \\
  \end{array}
\right]
\left[
  \begin{array}{cc}
    e^{\gamma_{1}a_{2}} & 0 \\
    0 & e^{\gamma_{2}a_{2}} \\
  \end{array}
\right]
\left[
  \begin{array}{cc}
    e^{-\gamma_{1}a_{1}} & 0 \\
    0 & e^{-\gamma_{2}a_{1}} \\
  \end{array}
\right]
\left[
  \begin{array}{cc}
    1 & 1 \\
    \gamma_{1} & \gamma_{2} \\
  \end{array}
\right]^{-1}\right.\\
&\left.\times
\left[
  \begin{array}{c}
    (1-p)a_{1}-q-K(1) \\
    1-p \\
  \end{array}
\right]
+
\left[
  \begin{array}{c}
    pa_{2}+q \\
    p \\
    a_{2}-K(2) \\
    1 \\
  \end{array}
\right]\right)\\
\doteq& F_{1}(a_{1},a_{2}).
\end{aligned}
\end{equation}
Note that in the above equation (\ref{SF-A1 A2 A3 A4 form-1}), only $a_{1},a_{2}$ are involved
to represent $(A_{1},A_{2},A_{3},A_{4})^{\top}$.

On the other hand, it follows from (\ref{SF-v1-b1-1}), (\ref{SF-v1-b1-2}),
(\ref{SF-v2-b1-1}), (\ref{SF-v2-b1-2}) that
\begin{equation*}
\begin{aligned}
\left[
  \begin{array}{c}
    A_{1} \\
    A_{2} \\
    A_{3} \\
    A_{4} \\
  \end{array}
\right]
=&\left[
  \begin{array}{cccc}
    e^{-\beta_{1}b_{1}} & 0 & 0 & 0 \\
    0 & e^{-\beta_{2}b_{1}} & 0 & 0 \\
    0 & 0 & e^{-\beta_{3}b_{1}} & 0 \\
    0 & 0 & 0 & e^{-\beta_{4}b_{1}} \\
  \end{array}
\right]
\left[
    \begin{array}{cccc}
      1 & 1 & 1 & 1 \\
      \beta_{1} & \beta_{2} & \beta_{3} & \beta_{4} \\
      \rho_{1} & \rho_{2} & \rho_{3} & \rho_{4} \\
      \beta_{1}\rho_{1} & \beta_{2}\rho_{2} & \beta_{3}\rho_{3} & \beta_{4}\rho_{4} \\
    \end{array}
  \right]^{-1}\\
&\times\left(
\left[
  \begin{array}{cc}
    0 & 0 \\
    0 & 0 \\
    1 & 1 \\
    \widetilde{\gamma}_{1} & \widetilde{\gamma}_{2} \\
  \end{array}
\right]
\left[
  \begin{array}{cc}
    e^{\widetilde{\gamma}_{1}b_{1}} & 0 \\
    0 & e^{\widetilde{\gamma}_{2}b_{1}} \\
  \end{array}
\right]
\left[
  \begin{array}{c}
    \widetilde{C}_{1} \\
    \widetilde{C}_{2} \\
  \end{array}
\right]
+
\left[
  \begin{array}{c}
    b_{1}-\widetilde{K}(1) \\
    1 \\
    \widetilde{p}b_{1}+\widetilde{q} \\
    \widetilde{p} \\
  \end{array}
\right]\right).
\end{aligned}
\end{equation*}
From (\ref{SF-v2-b2-1}) and (\ref{SF-v2-b2-2}), we have
\begin{equation}\label{SF-widetilde C1 widetilde C2}
\begin{aligned}
\left[
  \begin{array}{c}
    \widetilde{C}_{1} \\
    \widetilde{C}_{2} \\
  \end{array}
\right]
=
\left[
  \begin{array}{cc}
    e^{-\widetilde{\gamma}_{1}b_{2}} & 0 \\
    0 & e^{-\widetilde{\gamma}_{2}b_{2}} \\
  \end{array}
\right]
\left[
  \begin{array}{cc}
    1 & 1 \\
    \widetilde{\gamma}_{1} & \widetilde{\gamma}_{2} \\
  \end{array}
\right]^{-1}
\left[
  \begin{array}{c}
    (1-\widetilde{p})b_{2}-\widetilde{q}-\widetilde{K}(2) \\
    1-\widetilde{p} \\
  \end{array}
\right].
\end{aligned}
\end{equation}
So we obtain
\begin{equation}\label{SF-A1 A2 A3 A4 form-2}
\begin{aligned}
\left[
  \begin{array}{c}
    A_{1} \\
    A_{2} \\
    A_{3} \\
    A_{4} \\
  \end{array}
\right]
=&\left[
  \begin{array}{cccc}
    e^{-\beta_{1}b_{1}} & 0 & 0 & 0 \\
    0 & e^{-\beta_{2}b_{1}} & 0 & 0 \\
    0 & 0 & e^{-\beta_{3}b_{1}} & 0 \\
    0 & 0 & 0 & e^{-\beta_{4}b_{1}} \\
  \end{array}
\right]
\left[
    \begin{array}{cccc}
      1 & 1 & 1 & 1 \\
      \beta_{1} & \beta_{2} & \beta_{3} & \beta_{4} \\
      \rho_{1} & \rho_{2} & \rho_{3} & \rho_{4} \\
      \beta_{1}\rho_{1} & \beta_{2}\rho_{2} & \beta_{3}\rho_{3} & \beta_{4}\rho_{4} \\
    \end{array}
  \right]^{-1}\\
&\times\left(
\left[
  \begin{array}{cc}
    0 & 0 \\
    0 & 0 \\
    1 & 1 \\
    \widetilde{\gamma}_{1} & \widetilde{\gamma}_{2} \\
  \end{array}
\right]
\left[
  \begin{array}{cc}
    e^{\widetilde{\gamma}_{1}b_{1}} & 0 \\
    0 & e^{\widetilde{\gamma}_{2}b_{1}} \\
  \end{array}
\right]
\left[
  \begin{array}{cc}
    e^{-\widetilde{\gamma}_{1}b_{2}} & 0 \\
    0 & e^{-\widetilde{\gamma}_{2}b_{2}} \\
  \end{array}
\right]
\left[
  \begin{array}{cc}
    1 & 1 \\
    \widetilde{\gamma}_{1} & \widetilde{\gamma}_{2} \\
  \end{array}
\right]^{-1}\right.\\
&\left.\times
\left[
  \begin{array}{c}
    (1-\widetilde{p})b_{2}-\widetilde{q}-\widetilde{K}(2) \\
    1-\widetilde{p} \\
  \end{array}
\right]
+
\left[
  \begin{array}{c}
    b_{1}-\widetilde{K}(1) \\
    1 \\
    \widetilde{p}b_{1}+\widetilde{q} \\
    \widetilde{p} \\
  \end{array}
\right]\right)\\
\doteq& F_{2}(b_{1},b_{2}).
\end{aligned}
\end{equation}
Note that in the above equation (\ref{SF-A1 A2 A3 A4 form-2}), only $b_{1},b_{2}$ are involved
to represent $(A_{1},A_{2},A_{3},A_{4})^{\top}$.

Combining (\ref{SF-A1 A2 A3 A4 form-1}) and (\ref{SF-A1 A2 A3 A4 form-2}) leads to
\begin{equation}\label{SF-a1 a2 b1 b2}
\begin{aligned}
F_{1}(a_{1},a_{2})=F_{2}(b_{1},b_{2}),
\end{aligned}
\end{equation}
from which we can solve the threshold levels $a_{1},a_{2},b_{1},b_{2}$.
Then, $A_{1},A_{2},A_{3},A_{4}$ can be represented by
(\ref{SF-A1 A2 A3 A4 form-1}) or (\ref{SF-A1 A2 A3 A4 form-2}),
$C_{1},C_{2}$ can be represented by (\ref{SF-C1 C2}),
$\widetilde{C}_{1},\widetilde{C}_{2}$ can be represented
by (\ref{SF-widetilde C1 widetilde C2}).

Based on the verification theorem (Theorem \ref{VT theorem}),
we have the following theorem.
\begin{theorem}\label{theorem SF}
Suppose that the system of algebraic equations (\ref{SF-v1-a1-1})-(\ref{SF-v2-b2-2})
has a solution $a_{1}$, $a_{2}$, $b_{1}$, $b_{2}$, $A_{1}$, $A_{2}$, $A_{3}$, $A_{4}$,
$C_{1}$, $C_{2}$, $\widetilde{C}_{1}$, $\widetilde{C}_{2}$ such that $a_{1}<a_{2}<b_{1}<b_{2}$.
Let $v(x,i)$, $i\in\{1,2\}$, be given by Table \ref{Table} and satisfy the conditions of
Theorem \ref{VT theorem}. Define $(\tau_{1}^{*},\tau_{2}^{*})$ as follows:
\begin{equation*}
\begin{aligned}
\tau_{1}^{*}=\inf\{t\geq0|(X_{t},\alpha_{t})\notin D_{1}\},
\end{aligned}
\end{equation*}
and
\begin{equation*}
\begin{aligned}
\tau_{2}^{*}=\inf\{t\geq0|(X_{t},\alpha_{t})\notin D_{2}\},
\end{aligned}
\end{equation*}
where
\begin{equation*}
\begin{aligned}
D_{1}=\{(x,1)|x\in(a_{1},\infty)\}\cup\{(x,2)|x\in(a_{2},\infty)\},
\end{aligned}
\end{equation*}
and
\begin{equation*}
\begin{aligned}
D_{2}=\{(x,1)|x\in(-\infty,b_{1})\}\cup\{(x,2)|x\in(-\infty,b_{2})\}.
\end{aligned}
\end{equation*}
Then, $(\tau_{1}^{*},\tau_{2}^{*})$ is a Nash equilibrium for the two players
and $v(x,i)$, $i\in\{1,2\}$, is the corresponding value function of the game.
\end{theorem}

\section{Numerical results}\label{section NR}

In this section, we numerically demonstrate the dependence of threshold
levels on various model parameters and present a reduction to the case
with no regime switching.

\subsection{Dependence of threshold levels}

We take $r=3$, $\sigma(1)=2$, $\sigma(2)=4$, $K(1)=2$, $K(2)=3$,
$\widetilde{K}(1)=5$, $\widetilde{K}(2)=6$, $\lambda_{1}=2$, $\lambda_{2}=5$
as a set of \emph{benchmark parameters} and compute the closed-form solutions
given by Theorem \ref{theorem SF}. In this case, from (\ref{SF-a1 a2 b1 b2}),
the threshold levels are computed to be $(a_{1},a_{2},b_{1},b_{2})=(0.68,1.86,5.79,8.71)$.
The value function $v(x,i)$, $i\in\{1,2\}$, is plotted in Figure \ref{Figure 1}.
Moreover, Figures \ref{Figure 2} and \ref{Figure 3} verify that $v(x,i)$, $i\in\{1,2\}$,
satisfies the conditions of the verification theorem (Theorem \ref{VT theorem}).
Next, we examine the monotonicity of the threshold levels
$(a_{1},a_{2},b_{1},b_{2})$ with respect to the model parameters
$\sigma(i)$, $K(i)$, $\widetilde{K}(i)$, $i\in\mathcal{M}$.

\begin{figure}[htbp]
\centering
\includegraphics[width=5in]{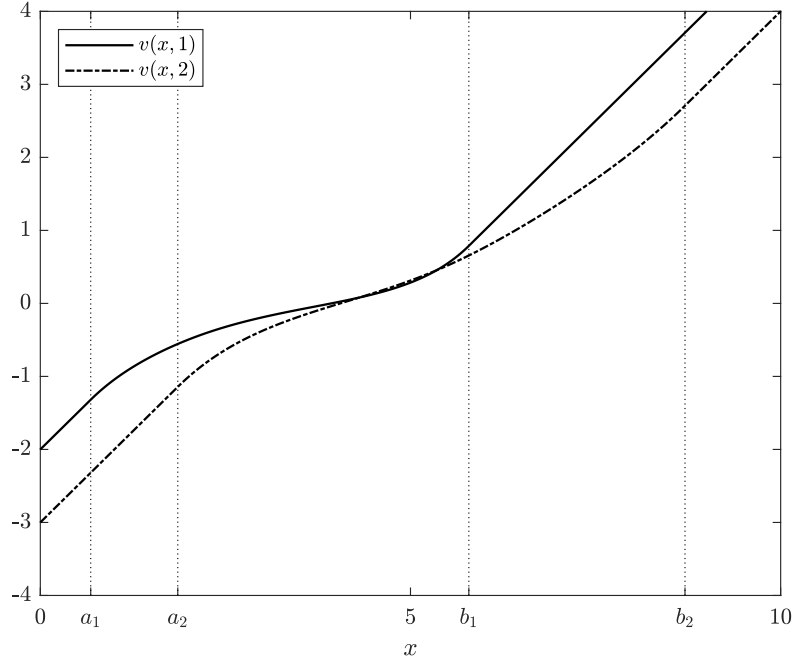}
\caption{Value function $v(x,i)$, $i\in\{1,2\}$}
\label{Figure 1}
\end{figure}

\begin{figure}[htbp]
\centering
\includegraphics[width=5in]{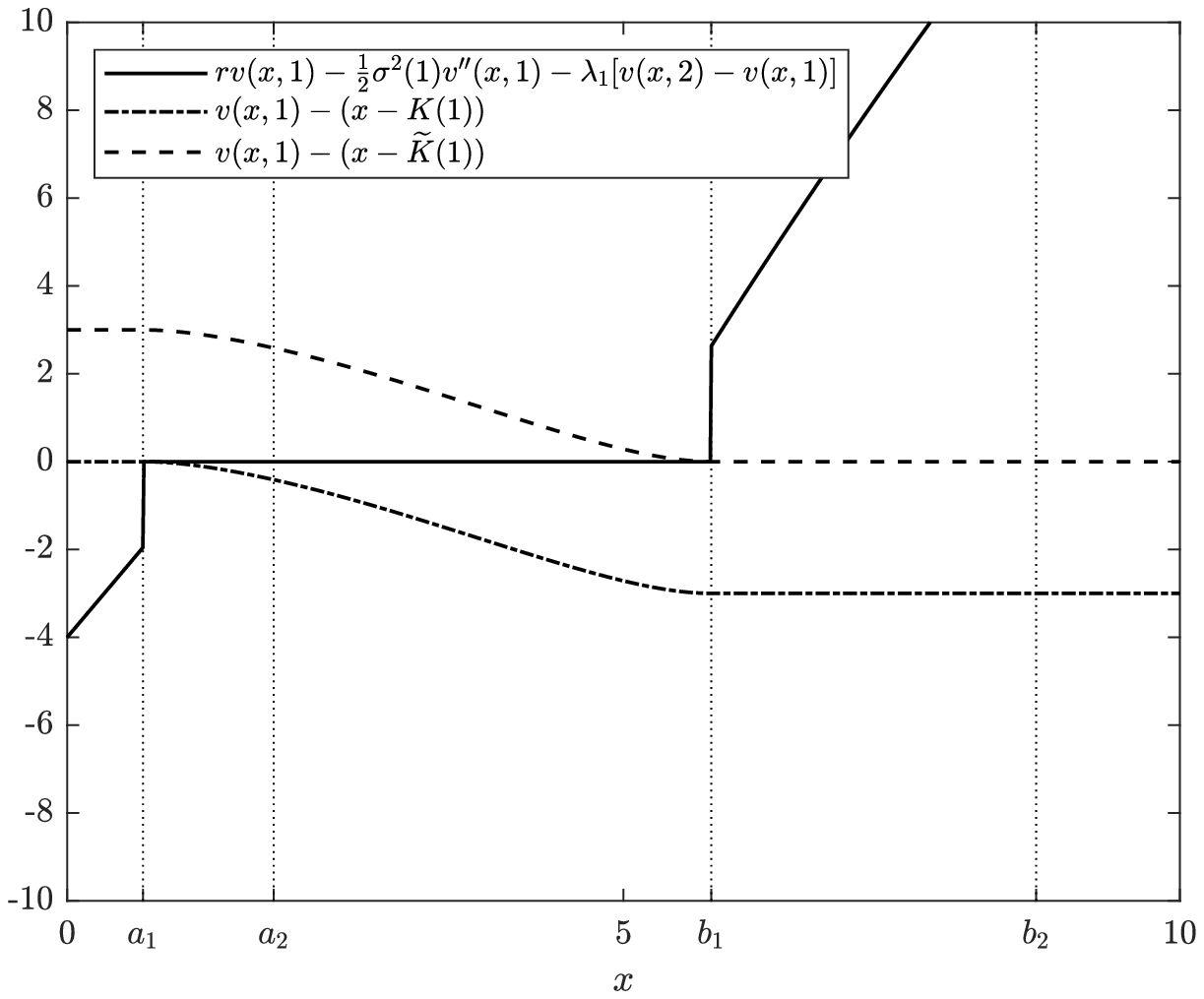}
\caption{Verification theorem for $v(x,1)$}
\label{Figure 2}
\end{figure}

\begin{figure}[htbp]
\centering
\includegraphics[width=5in]{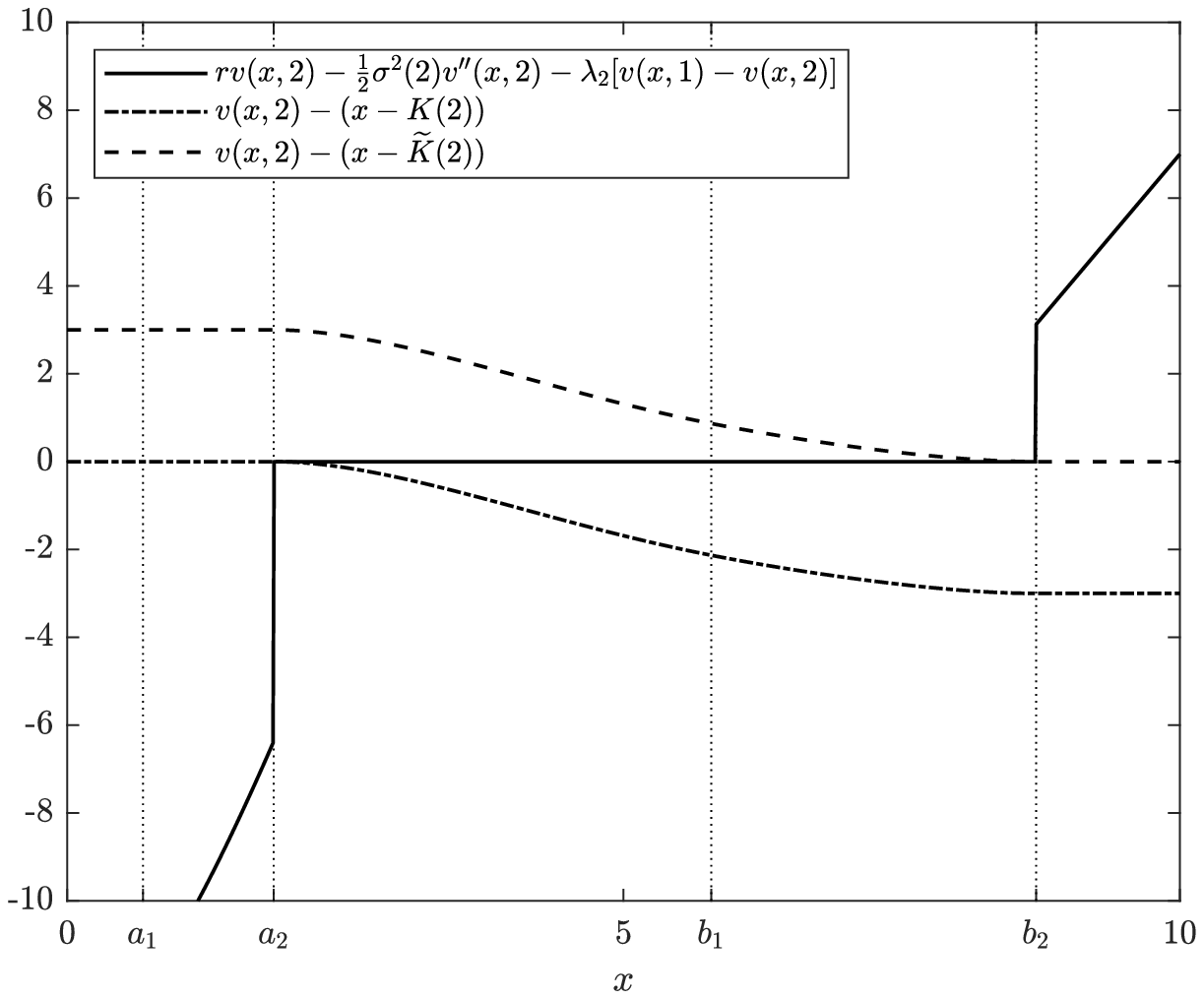}
\caption{Verification theorem for $v(x,2)$}
\label{Figure 3}
\end{figure}

First, we vary $\sigma(1)$ and keep all other parameters fixed.
The resulting $(a_{1},a_{2},b_{1},b_{2})$ are listed in Table \ref{sigma 1}.
From Table \ref{sigma 1}, we see that $a_{1}$ decreases and $b_{1}$ increases
as $\sigma(1)$ increases, while $a_{2}$ and $b_{2}$ remain nearly unchanged.
This is because $\sigma(1)$ is the volatility of $X_{t}$ when the
Markov chain $\alpha_{t}$ is at regime 1. The larger of $\sigma(1)$,
the farther $X_{t}$ can reach, which leads to a lower $a_{1}$ for Player 1
and a higher $b_{1}$ for Player 2 to achieve a better payoff.
On the other hand, we vary $\sigma(2)$ and the corresponding results
are similar and listed in Table \ref{sigma 2}.
\begin{table}[htbp]
\centering
\caption{Dependence on $\sigma(1)$}
\label{sigma 1}
\renewcommand\arraystretch{1.5}
\setlength{\tabcolsep}{3mm}
\begin{tabular}{|c|ccccc|}
  \hline
  $\sigma(1)$ & 2.0 & 2.5 & 3.0 & 3.5 & 4.0 \\
  \hline
  $a_{1}$ & 0.68 & 0.51 & 0.34 & 0.17 & 0.01 \\
  \hline
  $a_{2}$ & 1.86 & 1.84 & 1.82 & 1.81 & 1.79 \\
  \hline
  $b_{1}$ & 5.79 & 5.95 & 6.11 & 6.27 & 6.43 \\
  \hline
  $b_{2}$ & 8.71 & 8.71 & 8.71 & 8.72 & 8.72 \\
  \hline
\end{tabular}
\end{table}
\begin{table}[htbp]
\centering
\caption{Dependence on $\sigma(2)$}
\label{sigma 2}
\renewcommand\arraystretch{1.5}
\setlength{\tabcolsep}{3mm}
\begin{tabular}{|c|ccccc|}
  \hline
  $\sigma(2)$ & 4.0 & 4.5 & 5.0 & 5.5 & 6.0 \\
  \hline
  $a_{1}$ & 0.68 & 0.68 & 0.68 & 0.67 & 0.67 \\
  \hline
  $a_{2}$ & 1.86 & 1.73 & 1.61 & 1.49 & 1.37 \\
  \hline
  $b_{1}$ & 5.79 & 5.80 & 5.80 & 5.81 & 5.82 \\
  \hline
  $b_{2}$ & 8.71 & 8.85 & 8.99 & 9.13 & 9.27 \\
  \hline
\end{tabular}
\end{table}

Then, we vary $K(1)$ and keep all other parameters fixed.
The resulting $(a_{1},a_{2},b_{1},b_{2})$ are listed in Table \ref{K 1}.
Table \ref{K 1} suggests that $a_{1}$ increases if $K(1)$ increases,
while $a_{2}$, $b_{1}$, $b_{2}$ have no obvious variation.
Note that $a_{1}-K(1)$ is the payoff for Player 1 if it stops
when the Markov chain is at regime 1.
Player 1 is the one who wants to minimize (\ref{payoff}),
so a larger $K(1)$ means a bigger stopping reward, which encourages Player 1
to lock down profit by stopping earlier (i.e., a higher $a_{1}$).
On the other hand, we vary $K(2)$ and the corresponding results are similar
and listed in Table \ref{K 2}.
\begin{table}[htbp]
\centering
\caption{Dependence on $K(1)$}
\label{K 1}
\renewcommand\arraystretch{1.5}
\setlength{\tabcolsep}{3mm}
\begin{tabular}{|c|ccccc|}
  \hline
  $K(1)$ & 2.0 & 2.1 & 2.2 & 2.3 & 2.4 \\
  \hline
  $a_{1}$ & 0.68 & 0.84 & 1.00 & 1.16 & 1.31 \\
  \hline
  $a_{2}$ & 1.86 & 1.85 & 1.84 & 1.83 & 1.82 \\
  \hline
  $b_{1}$ & 5.79 & 5.79 & 5.79 & 5.79 & 5.79 \\
  \hline
  $b_{2}$ & 8.71 & 8.71 & 8.71 & 8.71 & 8.71 \\
  \hline
\end{tabular}
\end{table}
\begin{table}[htbp]
\centering
\caption{Dependence on $K(2)$}
\label{K 2}
\renewcommand\arraystretch{1.5}
\setlength{\tabcolsep}{3mm}
\begin{tabular}{|c|ccccc|}
  \hline
  $K(2)$ & 3.0 & 3.1 & 3.2 & 3.3 & 3.4 \\
  \hline
  $a_{1}$ & 0.68 & 0.62 & 0.56 & 0.49 & 0.43 \\
  \hline
  $a_{2}$ & 1.86 & 1.97 & 2.08 & 2.19 & 2.30 \\
  \hline
  $b_{1}$ & 5.79 & 5.78 & 5.78 & 5.78 & 5.77 \\
  \hline
  $b_{2}$ & 8.71 & 8.71 & 8.71 & 8.70 & 8.70 \\
  \hline
\end{tabular}
\end{table}

Finally, we vary $\widetilde{K}(1)$ and keep all other parameters fixed.
The resulting $(a_{1},a_{2},b_{1},b_{2})$ are listed in Table \ref{widetilde K 1}.
Table \ref{widetilde K 1} implies that $b_{1}$ increases in $\widetilde{K}(1)$,
while $a_{1}$, $a_{2}$, $b_{2}$ fluctuate slightly. This is due to that
a larger $\widetilde{K}(1)$ means a bigger stopping cost for Player 2
when the Markov chain is at regime 1, which in turn needs to be compensated
by a higher stopping level (i.e., a higher $b_{1}$).
On the other hand, we vary $\widetilde{K}(2)$ and the corresponding results
are similar and listed in Table \ref{widetilde K 2}.
\begin{table}[htbp]
\centering
\caption{Dependence on $\widetilde{K}(1)$}
\label{widetilde K 1}
\renewcommand\arraystretch{1.5}
\setlength{\tabcolsep}{3mm}
\begin{tabular}{|c|ccccc|}
  \hline
  $\widetilde{K}(1)$ & 5.0 & 5.1 & 5.2 & 5.3 & 5.4 \\
  \hline
  $a_{1}$ & 0.68 & 0.68 & 0.68 & 0.68 & 0.68 \\
  \hline
  $a_{2}$ & 1.86 & 1.85 & 1.85 & 1.84 & 1.84 \\
  \hline
  $b_{1}$ & 5.79 & 5.89 & 6.00 & 6.11 & 6.21 \\
  \hline
  $b_{2}$ & 8.71 & 8.55 & 8.40 & 8.26 & 8.12 \\
  \hline
\end{tabular}
\end{table}
\begin{table}[htbp]
\centering
\caption{Dependence on $\widetilde{K}(2)$}
\label{widetilde K 2}
\renewcommand\arraystretch{1.5}
\setlength{\tabcolsep}{3mm}
\begin{tabular}{|c|ccccc|}
  \hline
  $\widetilde{K}(2)$ & 6.0 & 6.1 & 6.2 & 6.3 & 6.4 \\
  \hline
  $a_{1}$ & 0.68 & 0.68 & 0.68 & 0.68 & 0.68 \\
  \hline
  $a_{2}$ & 1.86 & 1.86 & 1.86 & 1.86 & 1.86 \\
  \hline
  $b_{1}$ & 5.79 & 5.79 & 5.79 & 5.78 & 5.78 \\
  \hline
  $b_{2}$ & 8.71 & 8.97 & 9.22 & 9.49 & 9.75 \\
  \hline
\end{tabular}
\end{table}

\subsection{Reduction}

In the case with no regime switching, the state process $X\in R$ is described by
\begin{equation*}
\begin{aligned}
X_{t}=x+\sigma W_{t},\quad t\geq 0,
\end{aligned}
\end{equation*}
where $\sigma$ is a positive constant.
The objective functional for Player 1 to minimize and Player 2 to maximize is given by
\begin{equation*}
\begin{aligned}
J(x;\tau_{1},\tau_{2})
=E[e^{-r\tau_{1}}(X_{\tau_{1}}-K)1_{\{\tau_{1}<\tau_{2}\}}
+e^{-r\tau_{2}}(X_{\tau_{2}}-\widetilde{K})1_{\{\tau_{2}\leq\tau_{1}\}}],
\end{aligned}
\end{equation*}
where $K$ and $\widetilde{K}$ are two constants with $K<\widetilde{K}$.
As the case with regime switching, we would like to find a threshold-type
Nash equilibrium consists of two levels $a<b$ such that Player 1 will stop
if $X_{t}$ falls below $a$ and Player 2 will stop if $X_{t}$ goes above $b$.
For convenience, we list the derivation sketch of determining the threshold
levels $a$ and $b$ and the corresponding value function $v(x)$ as follows,
which is also based on Theorem \ref{VT theorem} but with no regime switching.

Consider $v(x)$ on the continuation region $(a,b)$:
\begin{equation*}
\begin{aligned}
rv(x)-\frac{1}{2}\sigma^{2}v^{\prime\prime}(x)=0.
\end{aligned}
\end{equation*}
The solution is
\begin{equation*}
\begin{aligned}
v(x)=A_{1}e^{\beta_{1}x}+A_{2}e^{\beta_{2}x},
\end{aligned}
\end{equation*}
where $\beta_{1},\beta_{2}$ satisfy the following equation:
\begin{equation*}
\begin{aligned}
r-\frac{1}{2}\sigma^{2}\beta^{2}=0,
\end{aligned}
\end{equation*}
which has two real roots:
\begin{equation*}
\begin{aligned}
\beta_{1,2}=\pm\sqrt{\frac{2r}{\sigma^{2}}}.
\end{aligned}
\end{equation*}
On $(-\infty,a)$, Player 1 stops such that
\begin{equation*}
\begin{aligned}
v(x)=x-K.
\end{aligned}
\end{equation*}
On $(b,\infty)$, Player 2 stops such that
\begin{equation*}
\begin{aligned}
v(x)=x-\widetilde{K}.
\end{aligned}
\end{equation*}
By applying the smooth-fit principle to $v(x)$ at $a$,
\begin{align}
&a-K=A_{1}e^{\beta_{1}a}+A_{2}e^{\beta_{2}a},\label{SF-v-a-1}\\
&1=A_{1}\beta_{1}e^{\beta_{1}a}+A_{2}\beta_{2}e^{\beta_{2}a},\label{SF-v-a-2}
\end{align}
and at $b$,
\begin{align}
&A_{1}e^{\beta_{1}b}+A_{2}e^{\beta_{2}b}=b-\widetilde{K},\label{SF-v-b-1}\\
&A_{1}\beta_{1}e^{\beta_{1}b}+A_{2}\beta_{2}e^{\beta_{2}b}=1.\label{SF-v-b-2}
\end{align}
From (\ref{SF-v-a-1}) and (\ref{SF-v-a-2}), we have
\begin{equation*}
\begin{aligned}
\left[
  \begin{array}{c}
    a-K \\
    1 \\
  \end{array}
\right]
=\left[
   \begin{array}{cc}
     e^{\beta_{1}a} & e^{\beta_{2}a} \\
     \beta_{1}e^{\beta_{1}a} & \beta_{2}e^{\beta_{2}a} \\
   \end{array}
 \right]
 \left[
  \begin{array}{c}
    A_{1} \\
    A_{2} \\
  \end{array}
\right].
\end{aligned}
\end{equation*}
Then,
\begin{equation}\label{SF-A1 A2 form-1}
\begin{aligned}
\left[
  \begin{array}{c}
    A_{1} \\
    A_{2} \\
  \end{array}
\right]
=\left[
   \begin{array}{cc}
     e^{\beta_{1}a} & e^{\beta_{2}a} \\
     \beta_{1}e^{\beta_{1}a} & \beta_{2}e^{\beta_{2}a} \\
   \end{array}
 \right]^{-1}
\left[
  \begin{array}{c}
    a-K \\
    1 \\
  \end{array}
\right].
\end{aligned}
\end{equation}
From (\ref{SF-v-b-1}) and (\ref{SF-v-b-2}), we have
\begin{equation*}
\begin{aligned}
\left[
  \begin{array}{c}
    b-\widetilde{K} \\
    1 \\
  \end{array}
\right]
=\left[
   \begin{array}{cc}
     e^{\beta_{1}b} & e^{\beta_{2}b} \\
     \beta_{1}e^{\beta_{1}b} & \beta_{2}e^{\beta_{2}b} \\
   \end{array}
 \right]
 \left[
  \begin{array}{c}
    A_{1} \\
    A_{2} \\
  \end{array}
\right].
\end{aligned}
\end{equation*}
Then,
\begin{equation}\label{SF-A1 A2 form-2}
\begin{aligned}
\left[
  \begin{array}{c}
    A_{1} \\
    A_{2} \\
  \end{array}
\right]
=\left[
   \begin{array}{cc}
     e^{\beta_{1}b} & e^{\beta_{2}b} \\
     \beta_{1}e^{\beta_{1}b} & \beta_{2}e^{\beta_{2}b} \\
   \end{array}
 \right]^{-1}
\left[
  \begin{array}{c}
    b-\widetilde{K} \\
    1 \\
  \end{array}
\right].
\end{aligned}
\end{equation}
It follows from (\ref{SF-A1 A2 form-1}) and (\ref{SF-A1 A2 form-2}) that
\begin{equation}\label{SF-a b}
\begin{aligned}
\left[
   \begin{array}{cc}
     e^{\beta_{1}a} & e^{\beta_{2}a} \\
     \beta_{1}e^{\beta_{1}a} & \beta_{2}e^{\beta_{2}a} \\
   \end{array}
 \right]^{-1}
\left[
  \begin{array}{c}
    a-K \\
    1 \\
  \end{array}
\right]
=\left[
   \begin{array}{cc}
     e^{\beta_{1}b} & e^{\beta_{2}b} \\
     \beta_{1}e^{\beta_{1}b} & \beta_{2}e^{\beta_{2}b} \\
   \end{array}
 \right]^{-1}
\left[
  \begin{array}{c}
    b-\widetilde{K} \\
    1 \\
  \end{array}
\right],
\end{aligned}
\end{equation}
from which we can solve $a$ and $b$.
Then, $A_{1}$ and $A_{2}$ can be represented by
(\ref{SF-A1 A2 form-1}) or (\ref{SF-A1 A2 form-2}).

We consider the following two sets of parameters and numerically compute
the corresponding threshold levels and value functions.

Case (\romannumeral1). Let $r=3$, $\sigma=\sigma(1)=2$,
$K=K(1)=2$, $\widetilde{K}=\widetilde{K}(1)=5$.
In this case, from (\ref{SF-a b}), the threshold levels are computed to be
$(a^{{\rm(\romannumeral1)}},b^{{\rm(\romannumeral1)}})=(1.19,5.81)$.
The corresponding value function $v^{{\rm(\romannumeral1)}}(x)$
is plotted in Figure \ref{Figure 4}.

\begin{figure}[htbp]
\centering
\includegraphics[width=5in]{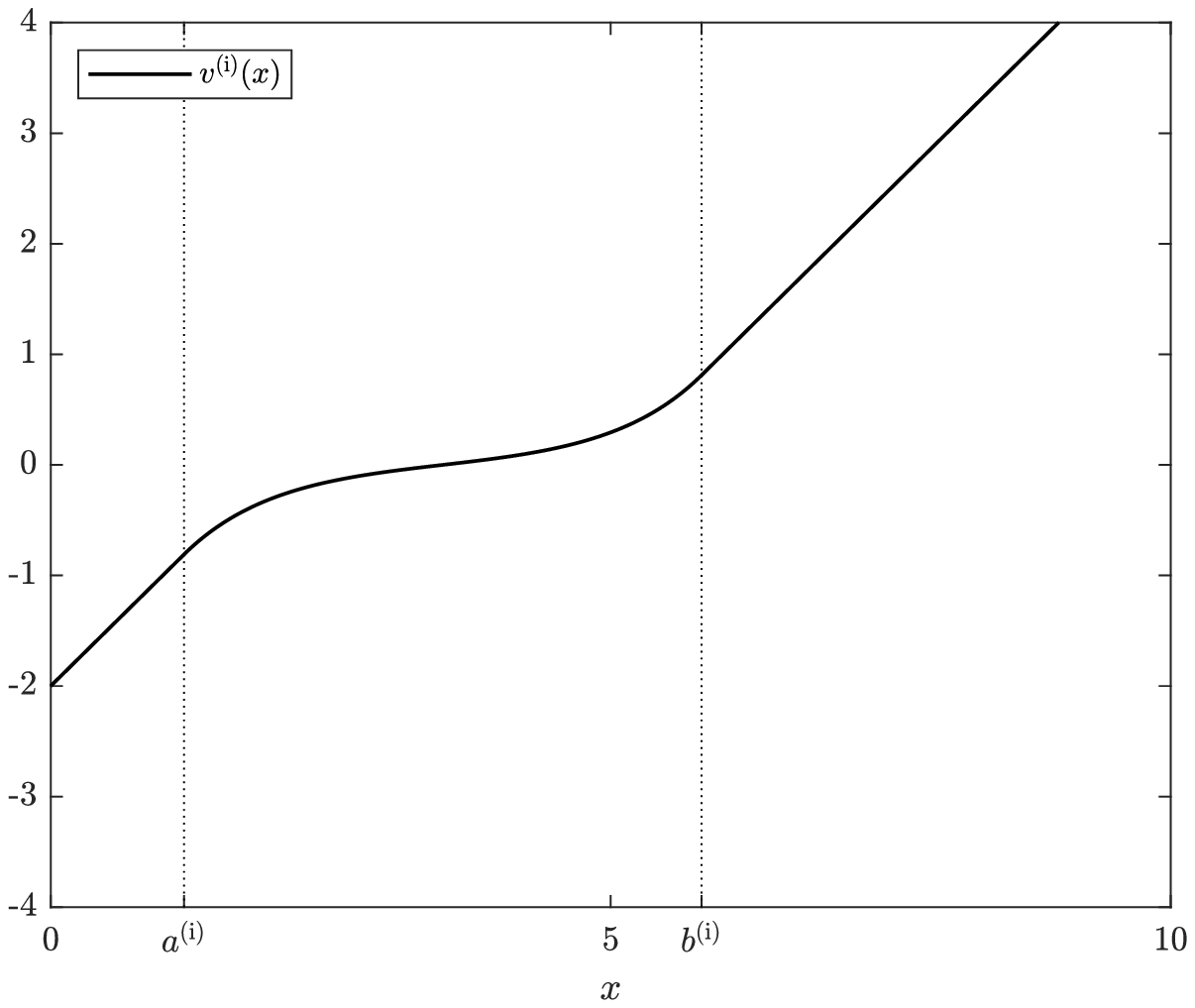}
\caption{Value function $v^{({\rm \romannumeral1})}(x)$}
\label{Figure 4}
\end{figure}

Case (\romannumeral2). Let $r=3$, $\sigma=\sigma(2)=4$,
$K=K(2)=3$, $\widetilde{K}=\widetilde{K}(2)=6$.
In this case, from (\ref{SF-a b}), the threshold levels are computed to be
$(a^{{\rm(\romannumeral2)}},b^{{\rm(\romannumeral2)}})=(1.44,7.56)$.
The corresponding value function $v^{{\rm(\romannumeral2)}}(x)$
is plotted in Figure \ref{Figure 5}.

\begin{figure}[htbp]
\centering
\includegraphics[width=5in]{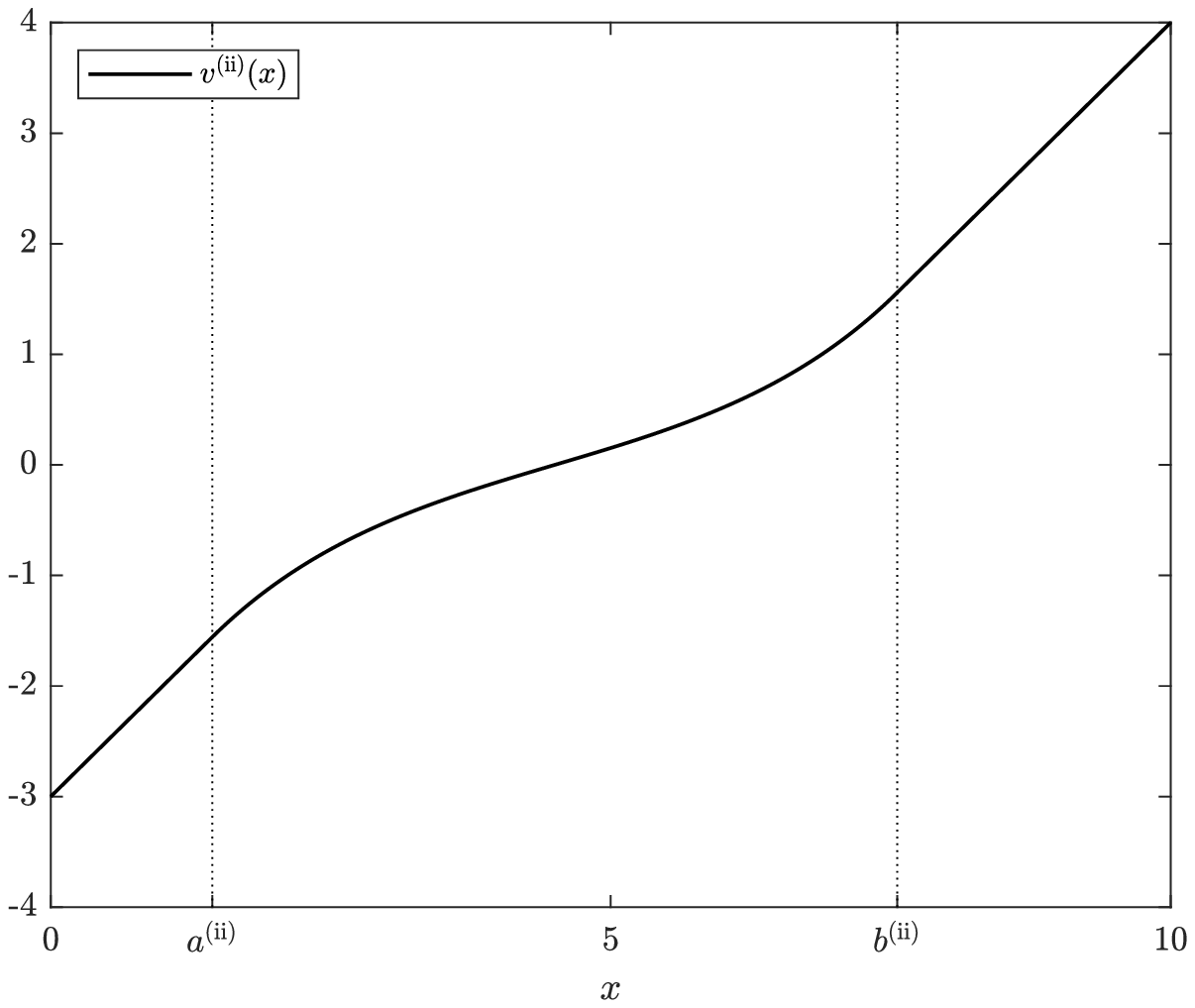}
\caption{Value function $v^{({\rm \romannumeral2})}(x)$}
\label{Figure 5}
\end{figure}

In addition, we can also verify numerically that
$v^{{\rm(\romannumeral1)}}(x)$ and $v^{{\rm(\romannumeral2)}}(x)$
satisfy the conditions of Theorem \ref{VT theorem}.
We omit the details for simplicity of presentation.

\section{Concluding remarks}\label{section CR}

There are three main contributions made in this paper.
Firstly, a verification theorem as a sufficient criterion
for Nash equilibriums is established, which involves a set
of VIs and an appropriate regularity requirement.
Secondly, the smooth-fit principle is further developed
for our stopping game problem with regime switching to solve
the VIs and derive a suitable system of algebraic equations.
Thirdly, various numerical experiments are included to
demonstrate the theoretical results with reasonable remarks.

This paper, we believe, has posed more questions than answers.
The problem formulation considered in this paper is a simple
but illustrative one, extensions to more general problems
may open up a new avenue for optimal stopping theory.
On the other hand, the stopping game problem with regime switching
should have a wide range of applications in many fields,
such as finance, management, engineering, and so on.
These topics in practice will be considered in our future study.

\end{document}